\documentclass[a4paper,10pt]{article}
\usepackage[T1]{fontenc}
\usepackage[cp1250]{inputenc}
\usepackage[english]{babel}
\usepackage{amsmath}
\usepackage{amsfonts}
\usepackage{amssymb}
\usepackage{mathrsfs}
\usepackage{amsthm}
\usepackage{graphicx}

\usepackage{color}
\usepackage{euscript}

\DeclareMathOperator{\dist}{dist}
\DeclareMathOperator{\BV}{\mathrm{BV}}
\newcommand{\BMO}{\mathrm{BMO}}
\newcommand{\eps}{\varepsilon}

\DeclareMathOperator{\E}{\mathbb{E}}
\DeclareMathOperator{\I}{I}
\DeclareMathOperator{\J}{J}
\newcommand{\W}{\mathfrak{W}}

\renewcommand{\leq}{\leqslant}
\renewcommand{\geq}{\geqslant}

\newcommand{\F}{\mathcal{F}}
\newcommand{\T}{\mathcal{T}}
\newcommand{\AF}{\mathcal{AF}}
\DeclareMathOperator{\Co}{Co}
\DeclareMathOperator{\Fl}{Fl}
\newcommand{\dH}{\dim_{\mathrm{H}}}

\newcommand{\R}{\mathbb{R}}
\DeclareMathOperator{\Div}{div}

\newtheorem{Le}{Lemma}[section]
\newtheorem{Def}{Definition}[section]

\newtheorem{Th}{Theorem}[section]
\newtheorem{Cor}[Le]{Corollary}
\newtheorem{Rem}[Le]{Remark}
\newtheorem{Conj}{Conjecture}
\numberwithin{equation}{section}

\begin{document}

\title{Martingale approach to Sobolev embedding theorems}

\author{Rami Ayoush
\and 
Dmitriy Stolyarov\thanks{Supported by Russian Foundation for Basic Research grant no. 18-31-00037}
\and
Michal Wojciechowski
}

\maketitle


\begin{abstract}
We prove a martingale analog of van Schaftingen's theorem and give sharp estimates on the lower Hausdorff dimension of measures in martingale shift invariant spaces. We also provide martingale analogs of trace theorems for Sobolev functions.
\end{abstract}

\section{Introduction}
Martingale analogs are ubiquitous in harmonic analysis and often serve effectively as models of theorems on Euclidean spaces. For example, it is a commonplace to treat martingale transforms as models of a Calder\'on--Zygmund operator. There are numerous examples in the literature. In \cite{Janson}, Janson characterized~$H_1$ martingales adapted to an~$m$-regular filtration in terms of boundedness of specific martingale transforms. His theorem might be thought of as a martingale analog of the celebrated Fefferman--Stein theorem on characterization of~$H_1$ by the Riesz transforms. Janson's approach revealed that antisymmetry is a necessary condition for a vector-valued multiplier to describe $H_1(\R^d)$. Later on Uchiyama in~\cite{U} proved sufficiency of this assumption. The martingale model allowed Janson to get rid of special structure of the multiplier, which lead to a general form of the Fefferman--Stein theorem in the Euclidean setting, proved by Uchiyama. In this paper we will employ the same principle to start the investigation of several problems related to Sobolev and~$\mathrm{BV}$-type spaces.

The original approach~\cite{Sobolev} to Sobolev embedding theorems is based on the Hardy--Littlewood--Sobolev inequality. Unfortunately, the latter inequality becomes invalid in the limit case~$p=1$. On the other hand, the embedding
\begin{equation}\label{GagliardoNirenberg}
\dot{W}_1^1(\mathbb{R}^d)\hookrightarrow L_{\frac{d}{d-1}}
\end{equation}
holds true, as it was proved later by Gagliardo~\cite{Gagliardo} and Nirenberg~\cite{Nirenberg}. The simple explanation of this fact is that the gradients of~$\dot{W}_1^1$-functions may be naturally embedded into the space~$L_1(\mathbb{R}^d,\mathbb{R}^d)$ of~$\mathbb{R}^d$-valued summable functions, however, they do not span the whole space~$L_1$. In fact, one can go the reverse direction and use the embedding~\eqref{GagliardoNirenberg} to prove that~$\dot{W}_1^1$ is not isomorphic to~$L_1$ as a Banach space, see~\cite{Kislyakov}. Later, the limiting embedding~\eqref{GagliardoNirenberg} was generalized to the case of higher order Sobolev spaces, fractional derivatives, and anisotropic case (we mention the papers~\cite{BesovIliin},~\cite{Kolyada},~\cite{Solonnikov} among others).

The original Sobolev's approach raises a natural question: for which $L_1$-based spaces is the Hardy--Littlewood--Sobolev inequality true? The embedding theorems cited above provide examples of such spaces (and the space of gradients of~$\dot{W}_1^1$ functions is the first among them). It is natural to assume that our space of functions is shift-invariant (we may shift the function without changing its norm) and also dilation invariant in a certain sense. The second assumption is more delicate since there are various groups of dilations acting on the Euclidean space (as we have mentioned, the limiting embedding theorems are also true in the anisotropic setting), and it is difficult to choose a general definition. For the sake of brevity, let us concentrate on the classical isotropic case. It seems that vector-valued functions provide good formalism. We will consider functions in~$L_1(\mathbb{R}^d,\mathbb{C}^\ell)$, i.e. summable functions in~$d$ variables taking values in~$\mathbb{C}^\ell$. The notation~$G(\ell,k)$ denotes the Grassmanian, i.e. the collection of all~$k$-dimensional linear subspaces of~$\mathbb{C}^\ell$. 
\begin{Def}\label{OurSobolev}
Let~$\Omega\colon S^{d-1}\to G(\ell,k)$ be a smooth function. Define the space~$\mathfrak{W}$ by the rule
\begin{equation*}
\mathfrak{W} = \Big\{g\in L_1(\mathbb{R}^d,\mathbb{C}^\ell)\,\Big|\; \forall \xi \in \mathbb{R}^d\setminus \{0\} \quad \hat{g}(\xi)\in \Omega\Big(\frac{\xi}{|\xi|}\Big)\Big\}.
\end{equation*}
\end{Def}
Here~$S^{d-1}$ is the unit sphere in~$\mathbb{R}^{d}$ and~$\hat{f}$ is the Fourier transform of~$f$ (the specification of factors in the definition of the Fourier transform is not of crucial importance here).
\begin{Conj}\label{MainConjecture}
Let~$\alpha \in (0,d)$. The Riesz potential~$\I_{\alpha}$ maps~$\mathfrak{W}$ to~$L_{\frac{d}{d-\alpha}}$ if and only if
\begin{equation}\label{Cancellation}
\bigcap_{\xi \in S^{d-1}}\Omega(\xi) = \{0\}.
\end{equation}
\end{Conj}
Consider a version of the space~$\BV$ of functions of bounded variation:
\begin{equation}\label{OurBV}
\mathfrak{BV} = \Big\{\mu\ \hbox{is an $\mathbb{C}^\ell$-valued measure of bounded variation}\,\Big|\; \forall \xi \in \mathbb{R}^d\setminus \{0\} \quad \hat{\mu}(\xi)\in \Omega\Big(\frac{\xi}{|\xi|}\Big)\Big\}.
\end{equation}
Conjecture~\ref{MainConjecture} claims that the proper analog of the Sobolev embedding holds if and only if~$\mathfrak{BV}$ does not contain~$\mathbb{C}^\ell$-valued delta-measures. Note that delta measures are the extreme points of the unit ball of the space of measures. The conjecture says that if we exclude the extreme points from our space of functions, then the Hardy--Littlewood--Sobolev inequality becomes true. 

In fact, the case where~$\Omega$ is a rational function and~$\alpha \geq 1$, had been already considered by van Schaftingen in~\cite{vS}. 
\begin{Def}
	Let~$V$ and~$U$ be finite dimensional linear spaces. Suppose that~$A(D)$ is a homogenous differential operator of order~$r$ with constant coefficients on~$\R^{d}$ mapping~$V$-valued functions to~$U$-valued functions, that is
	\begin{equation*}
	A(D)u = \sum\limits_{\genfrac{}{}{0pt}{-2}{\alpha \in \mathbb{N}^{d},}{|\alpha| = r}} A_{\alpha}(\partial^{\alpha}u)
	\end{equation*}
	for~$u \in C^{\infty}(\R^{d}, V)$, where~$A_{\alpha} \in \mathcal{L}(V,U)$. We say that~$A(D)$ is cancelling if
	\begin{equation*}
	\bigcap_{\xi \in \R^{d} \setminus \{0\}} \mathbb{A}[V] = \{0\},
	\end{equation*}
	where~$\mathbb{A}$ is matrix-valued differential symbol corresponding to~$A$. If additionally for~$\xi \neq 0$ the mapping~$\mathbb{A}(\xi)$ is one-to-one, then we refer to~$A$ as to an elliptic operator.
\end{Def}
\begin{Th}[van Schaftingen's theorem]\label{vanthm}
Suppose that~$A(D)$ is a homogenous differential operator of order~$r$ on $\R^{d}$ from~$V$ to~$U$. Then the estimate
	\begin{equation*}
	\lVert D^{r-1} f \rVert_{L_{\frac{d}{d-1}}} \lesssim \lVert A(D)f \rVert_{L_{1}}
	\end{equation*}
	holds for any~$f \in C^{\infty}_{0}(\R^{d};V)$ if and only if~$A(D)$ is elliptic and cancelling.
\end{Th}
Apart from implying its classical prototype, this remarkable result leads to  the Korn--Sobolev inequality (\cite{S}). In fact, Theorem~\ref{vanthm} is a result of a long development starting with~\cite{BB} and~\cite{vanSchaftingen1}. Particular cases of cancelling operators were earlier considered in~\cite{BB1},~\cite{BB2}, and~\cite{LanzaniStein} (citation is far from being complete, see the survey~\cite{vS2}).

It is known that the classical Gagliardo--Nirenberg embedding~$\dot{W}_1^1\hookrightarrow L_{\frac{d}{d-1}}$ may be improved to embedding into the best possible Lorentz space~$L_{\frac{d}{d-1},1}$ (we refer the reader to~\cite{BerghLofstrom} for background on Lorentz and Besov spaces). Moreover, one can also ``improve the interpolation parameter'' of smoothness and embed into the Besov--Lorentz space~$B^{0,1}_{\frac{d}{d-1},1}$, see~\cite{Kolyada}. We suggest a stronger form of Conjecture~\ref{MainConjecture}.
\begin{Conj}\label{StrongMainConjecture}
Let~$\alpha \in (0,d)$. The Riesz potential~$\I_{\alpha}$ maps~$\mathfrak{W}$ to~$B^{0,1}_{\frac{d}{d-\alpha},1}$ if and only if
\begin{equation*}
\bigcap_{\xi \in S^{d-1}}\Omega(\xi) = \{0\}.
\end{equation*}
\end{Conj}

The cancelling operators also improve Hardy-type inequalities (see~\cite{BV} and~\cite{MCanc}). In particular, the following result from~\cite{BV} holds.
\begin{Th}
	\label{hardy}
	Suppose that~$A(D)$ is a homogenous differential operator of order~$r$ on $\R^{d}$ acting  from~$V$ to $U$. Let~$s \in \{1, 2, \ldots, \min\{r,d-1\}\}$. The estimate
	\begin{equation*}
	\int\limits_{\mathbb{R}^{d}} \frac{\lvert D^{r-s}f(x)\rvert}{\lvert x\rvert^{s}}\,dx \lesssim \lVert A(D)f \rVert_{L_{1}}
	\end{equation*}
	holds for all~$f \in C^{\infty}_{0}(\R^{d};V)$ if and only if~$A(D)$ is elliptic and cancelling.
\end{Th}

We also note that the isotropic nature of homogeneity is not of crucial importance here. See~\cite{KMS} and~\cite{Stolyarov} for anisotropic inequalities in the style of Theorem~\ref{vanthm}. The authors suppose that the nature of the effect discussed above has nothing to do with Euclidean space instruments such as the Newton--Leibniz, Stokes, or coarea formulas, but has a purely harmonic analytic explanation. To find it, we construct a probabilistic model (which is, in fact, very close to Janson's model from~\cite{Janson}), where there is no Euclidean space structure and prove the direct analog of Conjecture~\ref{StrongMainConjecture} there (see Theorem~\ref{vanSchaftingen} and inequality~\eqref{MainInequality} below). It is highly likely that the methods of the present paper can be directly adapted to the Euclidean setting and will lead to the proof of Conjecture~\ref{StrongMainConjecture}.

We leave the Sobolev embedding and its analogs for a while to study yet another property of Sobolev functions that emphasizes the difference between~$\dot{W}_1^1$ and~$L_1$. Consider the classical space~$\BV$. It is known that the (vector-valued) measure~$\nabla f$ has good geometric properties. We concentrate on a very simple principle, which is a direct consequence of the coarea formula for~$\BV$ functions. Define the lower Hausdorff dimension~$\dH$ of a measure~$\mu$ to be the infimum of~$\alpha$ such that there exists a Borel set~$F$ of Hausdorff dimension not exceeding~$\alpha$ such that~$\mu(F) \ne 0$. Then,
\begin{equation*}
\dH \nabla f \geq d-1, \quad f \in \BV(\mathbb{R}^d).
\end{equation*}
We believe that this result may be generalized to the setting of~$\mathfrak{BV}$-spaces defined above. We cite a result from~\cite{RW}, where condition~\eqref{Cancellation} plays the central role.

\begin{Th}[\cite{RW}]\label{rwthm} Let~$\Omega\colon S^{d-1} \to S^{d-1}$ be a mapping. Suppose that~$\Omega$ contains at least two linearly independent vectors in its image. If a finite vector-valued measure~$\mu$ satisfies~$\hat{\mu}(\xi) \parallel \Omega(\frac{\xi}{\lvert \xi \rvert})$ for~$\xi \neq 0$, then
	\begin{equation*}
	\dH(\mu) \geq 1.
	\end{equation*}
\end{Th}
The appearance of the cancelling condition is not a coincidence here: the method from~\cite{SW} shows a direct connection between the lower Hausdorff dimension estimates for functions in~$\mathfrak{BV}$ and embedding theorems. This connection is, in fact, based on a proper generalization of Frostman's lemma. We cite a theorem from~\cite{SW}, which hints that the isotropic nature of homogeneity is not of crucial importance here as well. In the theorem below,~$\partial_j$ denotes the differentiation with respect to~$j$-th coordinate.
\begin{Th}[\cite{SW}]\label{swthm} 
If a vector measure $\mu = (\partial^{m_{1}}_{1}f,\partial_{2}^{m_2}\dots,\partial^{m_{d}}_{d}f)$ is a measure of finite variation, then
	\begin{equation*}
	\dH(\mu) \geq d-1.
	\end{equation*}
\end{Th}
Here we repeat the approach of~\cite{SW} on the level of martingales, using more precise embedding theorem that we have for $m$-regular models (see Theorem~\ref{Uncertainty} below). From the point of view of Hausdorff dimension estimates, van Schaftingen's condition has also a simple geometric meaning: it prevents measures from concentrating on $(d-1)-$dimensional hyperplanes. This observation was developed in~\cite{AW} and also leads to a martingale result.

Finally, we also investigate a martingale analog of Mazya's trace theorem (see~\cite{Mazja}) in the last section (see Theorems~\ref{TraceTheoremTrivial} and~\ref{TraceTheorem} below).
\begin{Th}[\cite{Mazja}]\label{MazjaThm}
The space~$\dot{W}_1^1(\mathbb{R}^d)$ embeds into~$L_p(\nu)$ if and only if~$\nu(B_r(x))^{\frac{1}{p}} \lesssim r^{d-1}$ for any Euclidean ball~$B_r(x)$.
\end{Th}

We mention two very recent preprints~\cite{APHR} and~\cite{vSS} that provide new results in the classical Euclidean setting that are closely related to the material of the present paper. 

The authors are grateful to Fedor Nazarov for correcting a mistake in an early version of this paper.

\subsection{The martingale model}

In next subsections we explain how to translate previously mentioned results into the language of $m$-regular martingales. Let us begin with their definition. Suppose that ~$m \geq 3$ is a fixed natural number and~$\F = \{\F_n\}_{n \geq 0}$ is an~$m$-uniform filtration (i.e. each atom in~$\F_n$ is split into~$m$ atoms in~$\F_{n+1}$ having equal masses). The set of atoms of the algebra~$\F_n$ is denoted by~$\AF_n$. There is a natural tree structure on the set of all atoms: the atom~$\omega' \in \AF_{n+1}$ is a son of~$\omega \in \AF_{n}$ if~$\omega' \subset \omega$. In such a case,~$\omega$ is the parent of~$\omega$, and we denote the parent of~$\omega'$ by~$(\omega')^{\uparrow}$.

We may treat our probability space equipped with an~$m$-uniform filtration as a metric space~$\mathbb{T}$. The points of~$\mathbb{T}$ are infinite paths in the tree of atoms (we start from the whole space, then choose an atom in~$\AF_1$, then pass to one of its sons in~$\AF_2$, and so on). The distance between the two paths~$\gamma_1$ and~$\gamma_2$ is defined by the standard formula
\begin{equation*}
\dist(\gamma_1,\gamma_2) = m^{-d}, \quad d = \max\{n\mid \gamma_1(j) = \gamma_2(j) \hbox{ for all } j<n\}.
\end{equation*}  
With this metric,~$\mathbb{T}$ becomes a compact metric space, so we can introduce Hausdorff measures there and define the Hausdorff dimension of sets and measures.   

Consider the space
\begin{equation*}
V = \{v\in\mathbb{R}^m\mid \sum_{1}^m v_j = 0\}.
\end{equation*}
For a martingale~$F$ adapted to~$\F$, let~$\{f_n\}_{n \geq 1}$ be the sequence of its martingale differences:
\begin{equation*}
f_{n+1} = F_{n+1} - F_n.
\end{equation*} 
For each atom~$\omega \in \AF_n$, the function~$f_{n+1}$ attains at most~$m$ values on~$\omega$, and thus, might be identified with an element of~$V$. Namely, for any~$\omega$, there is a bijective map
\begin{equation*}
\J_{\omega}\colon [1..m]\to \{\omega' \in \AF_{n+1}\mid \omega = (\omega')^\uparrow\}
\end{equation*}
which is naturally extended to the mapping between~$V$ and  restrictions of all possible martingale differences~$f_{n+1}$ to~$\omega$. This extended map is also called~$\J_{\omega}$.
For each~$n$ and each~$\omega$, we fix~$\J_{\omega}$. 

If the martingale~$F$ is~$\mathbb{R}^{\ell}$-valued, then~$f_{n+1}|_{\omega}$ might be naturally identified with an element of~$V^{\ell}$ (we apply~$\J_{\omega}$ to each coordinate), here~$\omega \in\AF_n$. Here and in what follows the norm~$\|\cdot\|$ always denotes the Euclidean norm in~$\mathbb{R}^\ell$.

\subsection{Hardy--Littlewood--Sobolev and van Schaftingen's inequalities}
Our main object of study is the Riesz potential:
\begin{equation*}
\Big(\I_{\alpha}[F]\Big)_N = \sum\limits_{n=1}^N m^{-\alpha n}f_n; \quad \hbox{or simply\ } \I_\alpha[f] = \sum\limits_{n\geq 1} m^{-\alpha n}f_n.
\end{equation*} 
Clearly,~$\I_{\alpha}$ maps martingales to martingales. Quantitative bounds for norms are more interesting. 
\begin{Th}[Hardy--Littlewood--Sobolev]
	For any~$p \in (1,\infty)$ and any~$q \in (p,\infty)$, the operator~$\I_{\frac{q-p}{qp}}$ maps~$L_p$ to~$L_q$.
\end{Th}
\begin{Rem}
	The inequality is sharp in the sense that~$\alpha \geq \frac{q-p}{qp}$ provided~$\I_{\alpha}$ is~$L_p\to L_q$ continuous. To see this, we prove the inequality ``on a single scale'' (i.e. when all the martingale differences but one are zero):
	\begin{equation}\label{LocalEmbedding1}
	\|f_n\|_{L_q} = m^{-\frac{n}{q}}\Big(\sum\limits_{w\in \AF_n}|f_n(\omega)|^q\Big)^\frac1q \stackrel{\scriptscriptstyle \ell_p\hookrightarrow \ell_q}{\leq} m^{-\frac{n}{q}}\Big(\sum\limits_{w\in \AF_n}|f_n(\omega)|^p\Big)^\frac1p = m^{n(\frac{1}{p}-\frac1q)}\|f_n\|_{L_p}.
	\end{equation}
	Clearly, these inequalities are sharp (pick~$f_n$ supported on a single atom of~$\F_{n-1}$).
\end{Rem}
\begin{Rem}
	In the limit case~$p=1$, the operator~$\I_{\frac{p-1}{p}}$ maps~$L_1$ to the Lorentz space~$L_{p,\infty}$. 
\end{Rem}
We refer the reader to~\cite{BerghLofstrom} for background on Lorentz spaces. For the proof of the martingale Hardy--Littlewood--Sobolev inequality, see~\cite{Watari}.
\begin{Le}\label{CounterExampleToHLS}
	The operator~$\I_{\frac{p-1}{p}}$ does not map~$L_1$ to~$L_p$.
\end{Le}
\begin{proof}
	Consider a special vector~$v_{\delta}\in V$:
	\begin{equation}\label{DeltaVector}
	v_{\delta} = m\cdot(1,0,0,\ldots,0) - (1,1,1,\ldots,1) = (m-1,-1,-1,\ldots,-1).
	\end{equation}
	Consider a martingale~$F$ such that for any~$\omega \in \AF_{n}$, we have the identity
	\begin{equation*}
	f_{n+1}|_{\omega} = F_{n}(\omega)\cdot \J_{\omega}[v_{\delta}].
	\end{equation*}
	In other words, if~$h_n$ is given by the formula
	\begin{equation*}
	h_{n+1} = \sum\limits_{\omega \in \AF_n}\J_{\omega}[v_{\delta}]\cdot\chi_{\omega}, \quad \hbox{then} \quad F_n = \prod\limits_{i=1}^{n} (1+h_i) - 1,
	\end{equation*}
	and therefore,~$f_{n+1}=h_{n+1}F_n$. In particular, it is easy to see that~$\E |F_n| \leq 2$. On the other hand,~$\|f_n\|_{L_p} = m^{\frac{p-1}{p}n}$ and these functions have ``almost disjoint supports'' in the sense that
	\begin{equation*}
	\|\I_{\frac{p-1}{p}}[F]\|_{L_p}^p \asymp \sum\limits_{n=1}^{\infty}\|m^{-\frac{p-1}{p}n}f_n\|_{L_p}^p = \infty.
	\end{equation*}
\end{proof}
The heuristics behind the proof above is that the martingale~$F$ represents a delta-measure via formula~\eqref{MeasureGeneratesAMartingale} below, which is an extremal point of the unit ball in~$L_1$. It seems that if one excludes these ``bad'' points in a uniform way, then the Hardy--Littlewood--Sobolev inequality may become valid. There is an encouraging result.
\begin{Th}
	The operator~$\I_{\frac{p-1}{p}}$ maps~$H_1$ to~$L_{p,1}$.
\end{Th}
Here~$H_1$ is the Hardy space, a proper subspace of~$L_1$ consisting of those martingales~$F$, for which the maximal function~$F^*$ is summable. It is easy to see that our ``delta-measure'' constructed in the proof of Lemma~\ref{CounterExampleToHLS} does not lie in the Hardy space. On Euclidean spaces, any measure in the Hardy space is uniformly continuous with respect to the Lebesgue measure. We claim that this is way too strong and the Hardy--Littlewood--Sobolev inequality is true for any but delta measure (in a certain uniform sense). In particular, all fractal measures are welcome.  

Theorem~\ref{vanthm} says that, in the world of Euclidean spaces, the Riesz potentials act on some~$L_1$-based Sobolev spaces. Now we will try to translate this theorem (to be more precise, Conjectures~\ref{MainConjecture} and~\ref{StrongMainConjecture}), to the martingale world.

Let~$W$ be a subspace of~$V^{\ell}$. Consider the subspace~$\W$ of the~$\mathbb{R}^\ell$-valued martingale space~$L_1$:
\begin{equation*}
\W = \Big\{F \in L_1(\mathbb{R}^\ell)\,\Big|\; \forall n,\ \omega \in \AF_n \quad f_{n+1}|_\omega \in \J_{\omega}[W]\Big\}.
\end{equation*}
One should think of~$\W$ as of a Sobolev-type space of zero smoothness (see Section~\ref{S5} below). We introduce two structural conditions.

{\bf The first structural condition} is that~$W$ does not contain the non-zero vectors~$v\otimes a$ where~$v\in V$ and~$a\in \mathbb{R}^\ell$.
\begin{Th}[Janson \cite{Janson}]\label{Janson}
	If~$W$ satisfies the first structural condition, then~$\W = H_1$ in the sense that the norms in these two spaces are equivalent. 
\end{Th}
We have defined~$H_1$ as a space of scalar functions. The definition of the~$\mathbb{R}^\ell$-valued Hardy space is competely similar. Clearly, an~$\mathbb{R}^\ell$-valued function belongs to~$H_1$ if and only if each of its coordinates lies in the scalar~$H_1$. Our formulation of Theorem~\ref{Janson} slightly differs from the original since we prefer to describe the results in terms of the space~$W$ rather than the matrices (the space~$W$ is the image of matrices which generate martingale transforms in Janson's setting); one can find our definition with the space~$W$ in~\cite{ChaoJanson}.

Theorem~\ref{Janson} is a probabilistic analog of the Fefferman--Stein Theorem in~$\mathbb{R}^d$. To be more precise, the Fefferman--Stein Theorem describes the case of a very specific multiplier. The general Euclidean space case, which is a closer analog of Theorem~\ref{Janson}, was proved in~\cite{U}.

{\bf The second structural condition} is that~$W$ does not contain the non-zero vectors~$v\otimes a$ where~$v$ is of the form~\eqref{DeltaVector} in the sense that it has~$m-1$ equal coordinates. 

\begin{Rem}\label{ComplexScalars}
	Unlike Definition~\ref{OurSobolev} and formula~\eqref{OurBV}, we consider real-valued martingales here for simplicity of the notation. The case of complex-valued martingales does not differ much. Note that, however, the rank-one vectors~$v\otimes a$ considered in the second structural condition and formulas~\eqref{KappaDefinition},~\eqref{FormulaForDerivativeOfKappa} below should still have real-valued~$v$ only when we switch to complex scalars.
\end{Rem}

\begin{Th}\label{vanSchaftingen}
	If ~$W$ satisfies the second structural condition, then~$\I_{\frac{p-1}{p}}$ maps~$\W$ to~$L_{p,1}$ for any~$p > 1$.
\end{Th}
This theorem is an analog of Conjecture~\ref{MainConjecture} for functions on~$\mathbb{R}^d$. In fact, we will prove even stronger inequality
\begin{equation}\label{MainInequality}
\sum\limits_{n=1}^{\infty} m^{-\frac{p-1}{p}n}\|f_n\|_{L_{p,1}} \lesssim \|F\|_{\W},
\end{equation}
which immediately implies the theorem above (by the triangle inequality in~$L_{p,1}$). This inequality is a martingale analog of Conjecture~\ref{StrongMainConjecture}.

We will need some Besov spaces:
\begin{equation}\label{BesovSpace}
\|F\|_{B_{p}^{\beta,1}} = \sum\limits_{n=0}^{\infty}m^{\beta n}\|f_n\|_{L_p}.
\end{equation}
The theorem below also follows from~\eqref{MainInequality} since~$L_{p,1} \hookrightarrow L_{p,p} = L_p$.
\begin{Th}
	If ~$W$ satisfies the second structural condition, then~$I_{\frac{p-1}{p}}$ maps~$\W$ to~$B_{p}^{0,1}$ 
	for $p > 1$. In other words,~$\W \hookrightarrow B_{p}^{-\frac{p-1}{p},1}$.
\end{Th}

\subsection{An Uncertainty Principle}

Let us observe that each measure~$\mu$ generates a martingale via conditional expectation, that is by the formula
\begin{equation}\label{MeasureGeneratesAMartingale}
F_n = \sum\limits_{\omega\in \AF_n} \mu(\omega)\chi_{\omega}.
\end{equation}
Note that the characteristic functions of cylinders
\begin{equation*}
\{\gamma\mid \gamma(j) \hbox{ is fixed for all } j<n\}
\end{equation*} 
are continuous and they form a total family in the space~$C(\mathbb{T})$ (i.e. each measure on~$\mathbb{T}$ is uniquely defined by its values at cylinders). Using this fact one can finish the proof of one-to-one correspondence via~\eqref{MeasureGeneratesAMartingale} between finite measures on~$\mathbb{T}$ and~$L_1$-martingales adapted to~$\F$. 

We may consider all vector-valued measures~$\mu$ whose martingales lie in~$\W$. We claim that these measures cannot be too singular. We introduce some more notation to formulate our results. 

Let~$v\in V$ be a vector. Consider the function~$\kappa_v\colon[0,\infty)\to \mathbb{R}$ given by the formula
\begin{equation*}
\kappa_v(\theta) = \theta\log\Big(\frac{1}{m}\sum\limits_{j=1}^m|1+v_j|^{\frac{1}{\theta}}\Big) = \log\|{\bf 1}+ v\|_{L_{\frac{1}{\theta}}}.
\end{equation*}
By H\"older's inequality, this function is convex and non-increasing. If, in addition,~$v_j \geq -1$ for any~$j$, this function also satisfies~$\kappa_v(1) = 0$. Consider yet another function~$\kappa$:
\begin{equation}\label{KappaDefinition}
\kappa(\theta) = \sup\Big\{\kappa_v(\theta)\;\Big|\, \exists a\in\mathbb{R}^\ell\setminus\{0\} \hbox{ such that } v\otimes a\in W \hbox{ and } \forall j \quad v_j \geq -1\Big\}.
\end{equation}
The function~$\kappa$ is also convex, non-increasing and satisfies~$\kappa(1) = 0$.
\begin{Rem}\label{StrictInequality}
	The second structural condition in the previous subsection is equivalent to 
	\begin{equation*}
	\kappa\big(p^{-1}\big) < \frac{p-1}{p}\log m
	\end{equation*}
	for any~$p \in (1,\infty]$.
\end{Rem}

The theorem below provides a martingale analog of the results in~\cite{RW} and~\cite{SW}. We recall the definition of the lower Hausdorff dimension of a measure:
\begin{equation*}
\dH \mu = \{\inf \alpha \mid \exists F \ \hbox{such that}\ \mu(F) \ne 0, \dH F \leq \alpha\}.
\end{equation*}
\begin{Th}\label{Uncertainty}
	For any~$\mu \in \W$,
	\begin{equation}\label{If}
	\dH(\mu) \geq 1+\frac{\kappa'(1)}{\log m}.
	\end{equation}
	This inequality is sharp in the sense that for any choice of~$m$ and~$W$, there exists~$\mu$ such that this inequality turns into equality.
\end{Th}
As an elementary computation shows,
\begin{equation}\label{FormulaForDerivativeOfKappa}
\kappa'(1) = \inf\Big\{-\frac{1}{m}\sum\limits_{j=1}^m(1+v_j)\log(1+v_j)\;\Big|\, \exists a\in\mathbb{R}^\ell\setminus \{0\} \hbox{ such that } v\otimes a\in W \hbox{ and }
 \forall j \quad v_j \geq -1\Big\}.
\end{equation}

\subsection{More on $m$-regular model}\label{S5}
We enclose a small dictionary translating martingale theory to Euclidean space notions and theorems.  We also try to justify the first two analogies given below.

\begin{center}
	\begin{tabular}{|r|c|l|} \hline
		$m$-regular martingale & $\R^{d}$ & Remark
		\\ \hline
		first structural assumption & full antisymmetry of the multiplier & \\
		second structural assumption & cancelling condition & \\
		Theorem \ref{vanSchaftingen} & Theorem~\ref{vanthm}, Conjecture \ref{MainConjecture} &  H-L-S  inequality\\
		Theorem \ref{Uncertainty}  & Theorem \ref{rwthm}, Theorem \ref{swthm} &  dimension estimates\\
		Corollary \ref{antisymmetry} & Conjecture \ref{ascon} & partial antisymmetry \\ 
		Theorem \ref{TraceTheorem} & Theorem \ref{MazjaThm} & trace theorem \\ 
		\hline
	\end{tabular}
\end{center}

Let us equip the set~$[1..m]$ with a commutative group structure~$G$. Elements of~$V$ might be thought of as functions on~$G$ with zero mean value. It is more convenient to switch to complex scalars (see Remark~\ref{ComplexScalars}). Let us restrict ourselves to the case where the space~$W$ is shift invariant: if~$f\in V^{\ell}$ belongs to~$W$, then~$f(z\cdot)$ also belongs to~$W$ for any element~$z \in G$; we treat the elements of~$W$ as~$\mathbb{C}^\ell$-valued functions on~$G$. Such spaces can be described in terms of the Fourier transform:
\begin{equation}\label{ShiftInvariant}
W = \Big\{f\colon G\to\mathbb{C}^\ell\,\Big|\, \forall \gamma \in G\setminus\{0\}\quad \hat{f}(\gamma) \in W_{\gamma}\Big\},
\end{equation}
where~$W_{\gamma}$ are fixed subspaces of~$\mathbb{R}^\ell$.
\begin{Le}
	Let~$W$ be given by formula~\eqref{ShiftInvariant}. The first structural assumption is true if and only if~$W_{\gamma} \cap W_{-\gamma} = \{0\}$ for $\gamma \in G\setminus \{0\}$.
\end{Le}
The proof of this lemma is similar to the proof of the following lemma.
	

\begin{Le}
	Let~$W$ be given by formula~\eqref{ShiftInvariant}. The second structural assumption on~$W$ is equivalent to the cancellation condition
	\begin{equation}\label{FiniteVSCondition}
	\bigcap_{\gamma \in G\setminus \{0\}}W_{\gamma} = \{0\}.
	\end{equation}
\end{Le}
\begin{proof}
	Assume the second structural assumption does not hold. This means there exists a non-zero vector~$v\otimes a \in W$,~$v\in V$ and~$a\in\mathbb{C}^\ell$ such that~$v$ has~$m-1$ equal coordinates. By translation invariance of~$W$, the vector~$(m\delta_0 - {\bf 1}_{G})\otimes a$ also lies in~$W$. Therefore,~$a\in W_{\gamma}$ for all~$\gamma\in G\setminus\{0\}$ and~\eqref{FiniteVSCondition} is violated.
	
	Assume that~\eqref{FiniteVSCondition} does not hold. Let~$a$ be a non-zero vector that belongs to all~$W_{\gamma}$. Clearly,~$(m\delta_0 - {\bf 1}_{G})\otimes a \in W$ in this case, and the second structural assumption is violated.  
\end{proof}

\section{Linear algebraic part}

 \begin{Le}\label{InequalityForNon-Intensive2}
	Let~$p \in (1,\infty]$ be fixed. For any~$\delta > 0$, there exists a positive constant~$\eps_0$ such that for any~$\eps < \eps_0$ and any vectors~$a, b_1,b_2,\ldots, b_m$ in~$\mathbb{R}^\ell$ satisfying~$\{b_j\}_j \in W$ and
	\begin{equation}\label{NonIntensivityLinearAlgebra}
	\frac{1}{m}\sum\limits_{j=1}^m\|a+b_j\| - \|a\| \leq \eps \|a\|,
	\end{equation}
	there is the estimate
	\begin{equation*}
	\Big(\frac{1}{m}\sum\limits_{j=1}^m\|a+b_j\|^p\Big)^{\frac{1}{p}} \leq e^{\kappa(p^{-1})+\delta}\|a\|.
	\end{equation*}
\end{Le}
\begin{Rem}\label{JensenRemark}
	Note that the quantity on the left of~\eqref{NonIntensivityLinearAlgebra} is always non-negative since
	$\sum_j b_j = 0$, because~$\{b_j\}_j \in W \subset V^{\ell}$.
\end{Rem}

\begin{proof}[Proof of Lemma~\ref{InequalityForNon-Intensive2}.]
	Let~$\pi_a$ denote the orthogonal projection onto the line~$\mathbb{R}a \subset \mathbb{R}^\ell$, let~$\pi_{a^\perp}$ denote the projection onto its orthogonal complement. We start with a simple observation:
	\begin{equation*}
	\frac{1}{m}\sum\limits_{j=1}^m\|a+b_j\| \stackrel{\scriptscriptstyle \pi_a[a] = a}{\geq} \frac{1}{m}\sum\limits_{j=1}^m\|a + \pi_a[b_j]\| \geq \|a\|.
	\end{equation*}
	The difference between the first and the second quantities is 
	\begin{equation}\label{FirstDifference}
	\frac{1}{m}\sum\limits_{j=1}^m\big(\|a+b_j\| - \|a + \pi_a[b_j]\|\big) = \frac{1}{m}\sum\limits_{j=1}^m\frac{\|\pi_{a^\perp}[b_j]\|^2}{\|a+b_j\| + \|a + \pi_a[b_j]\|} \asymp \frac{1}{m}\sum\limits_{j=1}^m\frac{\|\pi_{a^\perp}[b_j]\|^2}{\|a+b_j\|}. 
	\end{equation}
	The difference between the second and the third quantities is
	\begin{equation}\label{SecondDifference}
	\begin{split}
	\frac{1}{m}\sum\limits_{j=1}^m(\|a + \pi_a[b_j]\| - \|a\|) =\frac{1}{m}\sum\limits_{j=1}^m\big(|a + \pi_a[b_j]| - (a+\pi_a[b_j])\big) = \\ \frac{2}{m}\sum\limits_{j=1}^m(-a-\pi_a[b_j])\chi_{\{a+\pi_a[b_j] < 0\}}.
	\end{split}
	\end{equation}
	Here we interpret~$a$ and~$\pi_a[b_j]$ as reals in such a way that~$a$ is the positive number~$\|a\|$.

	Assume the contrary. Let there exist a sequence~$\{\eps_n\}_{n}$ tending to zero such that for any~$n$ there exist vectors~$a_n$ and~$\{b_{nj}\}_{j=1}^m$,~$\{b_{nj}\}_j \in W$ such that
	\begin{equation*}
	\Big(\frac{1}{m}\sum\limits_{j=1}^m \|a_n+b_{nj}\|\Big) - \|a_n\| \leq \eps_n\|a_n\|,
	\end{equation*}
	but
	\begin{equation}\label{Contradictory}
	\liminf_n \frac{\Big(\frac{1}{m}\sum_{1}^m\|a_n+b_{nj}\|^p\Big)^{\frac{1}{p}}}{\|a_n\|} \geq e^{\kappa(p^{-1})+\delta}.
	\end{equation}
	
	Without loss of generality, we may assume~$\|a_n\| = 1$. We know that the quantity~\eqref{SecondDifference} tends to zero:
	\begin{equation*}
	\sum\limits_{j=1}^m(-1-\pi_{a_n}[b_{nj}])\chi_{\{1+\pi_{a_n}[b_{nj}] < 0\}} = o(1), \quad n\to \infty.
	\end{equation*}
	In particular, this means~$\pi_{a_n}[b_{nj}] \geq -1 + o(1)$ for any~$j$. On the hand,~$\sum_j \pi_{a_n}[b_{nj}] = 0$, so
	\begin{equation}\label{Projection1}
	-1 + o(1) \leq \pi_{a_n}[b_{nj}] \leq m-1 + o(1),\quad n\to\infty,
	\end{equation}
	for any~$j$.
	
	We also know that the quantity~\eqref{FirstDifference} tends to zero as~$n\to \infty$:
	\begin{equation*}
	\sum\limits_{j=1}^m\frac{\|\pi_{{a_n}^\perp}[b_{nj}]\|^2}{\|a_n+b_{nj}\|} = o(1),\quad n\to \infty.
	\end{equation*}
	We use the triangle inequality in the denominator and also~\eqref{Projection1}:
	\begin{equation*}
	\sum\limits_{j=1}^m\frac{\|\pi_{{a_n}^\perp}[b_{nj}]\|^2}{m+\|\pi_{a_n^\perp}[b_{nj}]\|} = o(1).
	\end{equation*}
	In particular,~$\pi_{{a_n}^\perp}[b_{nj}] = o(1)$ for any~$j$. Thus, by~\eqref{Projection1},~$\|b_{nj}\| \leq m-1 + o(1)$. Passing to a subsequence if needed, we may assume that~$b_{nj} \to b_j$ for each~$j$ and also~$a_n \to a$. Clearly, then~$b_j \parallel a$ for all~$j$, the vector~$\{b_j\}_j$ belongs to~$W$, and for all~$j$
	\begin{equation*}
	-1 \leq \pi_{a}[b_{j}].
	\end{equation*}
	This means~$\{b_j\} = v\otimes a$ with~$v_j \geq -1$ and (by~\eqref{Contradictory})
	\begin{equation*}
	\Big(\frac{1}{m}\sum\limits_{j=1}^m\|1+v_j\|^p\Big)^{\frac{1}{p}} \geq e^{\kappa(p^{-1})+\delta},
	\end{equation*}
	which contradicts~\eqref{KappaDefinition}.
\end{proof}

\section{Combinatorial part}
\begin{Def}\label{DefinitionOfConvexity}
	We say that an atom~$\omega \in \AF_{n}$ is~$\eps$-convex for the martingale~$F \in \W$ if
	\begin{equation*}
	\E(\|F_{n+1}\| - \|F_n\|)\chi_{\omega} \geq \eps \E\|F_n\|\chi_{\omega}.
	\end{equation*}
	If this inequality does not hold, then~$\omega$ is called~$\eps$-flat for~$F$. 
\end{Def}
Since~$F$ and~$\eps$ are fixed, we call atoms simply convex or flat. The set of all convex atoms is denoted by~$\Co$ and the set of all flat atoms is denoted by~$\Fl$.
\begin{Rem}\label{FlatLInfty}
	By Lemma~\ref{InequalityForNon-Intensive2}, 
	\begin{equation*}
	\|F_{n+1}\|_{L_p(\omega)} \leq e^{\kappa(p^{-1}) + o_{\eps\to 0,p}(1)}\|F_n\|_{L_p(\omega)}
	\end{equation*}
	for any~$\eps$-flat atom~$\omega \in \AF_n$.
\end{Rem}
Definition~\ref{DefinitionOfConvexity} leads to the natural splitting~$F = F_{\Co} + F_{\Fl}$, where
\begin{equation}\label{ConvexAndFlat}
F_{\Co} = \sum\limits_{n=0}^{\infty}\sum\limits_{\omega \in \Co \cap \AF_n}f_{n+1}\chi_\omega; \quad F_{\Fl} = \sum\limits_{n=0}^{\infty}\sum\limits_{\omega \in \Fl \cap \AF_n}f_{n+1}\chi_\omega.
\end{equation}
We also note a useful identity:
\begin{equation}\label{StepwiseSplitting}
\|F\|_{L_1} = \sum\limits_{n=0}^\infty \E\big(\|F_{n+1}\| - \|F_n\|\big) = \sum\limits_{n=0}^{\infty}\sum\limits_{\omega \in \AF_n}\E(\|F_{n+1}\| - \|F_n\|)\chi_{\omega}.
\end{equation}
By Remark~\ref{JensenRemark}, each summand in this double sum is non-negative.

\subsection{Convex atoms}
\begin{Le}\label{TriangleInequality}
	Let~$\omega \in \AF_n$ be a convex atom. Then,
	\begin{equation*}
	\E\|f_{n+1}\|\chi_{\omega} \lesssim_{\eps} \E(\|F_{n+1}\| - \|F_n\|)\chi_{\omega}.
	\end{equation*}
\end{Le}
\begin{proof}
	We restate the definition of the convexity of~$\omega$:
	\begin{equation*}
	\E\|F_n\|\chi_\omega \leq \frac{1}{\eps}\E(\|F_{n+1}\| - \|F_n\|)\chi_{\omega}.
	\end{equation*}
	Thus,
	\begin{equation*}
	\E\|F_{n+1}\|\chi_{\omega} \leq \frac{\eps + 1}{\eps}\E(\|F_{n+1}\| - \|F_n\|)\chi_{\omega}
	\end{equation*}
	and finally
	\begin{equation*}
	\E\|f_{n+1}\|\chi_\omega \leq \E\big(\|F_n\|+\|F_{n+1}\|\big)\chi_{\omega} \leq \frac{\eps + 2}{\eps}\E(\|F_{n+1}\| - \|F_n\|)\chi_{\omega}.
	\end{equation*}
\end{proof}
\begin{Cor}\label{ToBesov}
	We have a very nice bound
	\begin{equation*}
	\sum\limits_{n=0}^{\infty}\Big\|\sum\limits_{\omega \in \Co \cap \AF_n}f_{n+1}\chi_\omega\Big\|_{L_1} \lesssim \sum\limits_{n=0}^{\infty}\sum\limits_{\omega \in \Co \cap\AF_n}\E(\|F_{n+1}\| - \|F_n\|)\chi_{\omega} \stackrel{\scriptscriptstyle\eqref{StepwiseSplitting}}{\leq} \|F\|_{L_1}.
	\end{equation*}
	In other words, we have proved
	\begin{equation*}
	\|F_{\Co}\|_{B_1^{0,1}} \lesssim \|F\|_{L_1}, 
	\end{equation*}
	where~$F_{\Co}$ is given by~\eqref{ConvexAndFlat}.
\end{Cor}

\subsection{Estimate for flat atoms}
We will need a tiny portion of combinatorics here. We connect two flat atoms~$\omega$ and~$\omega'$ by an edge if either~$\omega =  (\omega')^\uparrow$ or~$\omega' = (\omega)^\uparrow$. We get a graph whose vertices are flat atoms. Clearly, it has no cycles. Therefore, it is a union of trees (we split our graph into trees that are maximal by inclusion, i.e. they are connectivity components of our graph). Each tree~$\T$ has a root, which is a flat atom itself. This provides us with the decomposition
\begin{equation}\label{SplittingIntoTrees}
F_{\Fl} = \sum\limits_{\T}F_{\T}, \quad \hbox{where} \quad F_{\T} = \sum\limits_{n\geq 0}\sum\limits_{\omega \in \T\cap\AF_n}f_{n+1}\chi_\omega,
\end{equation}
and the summation in the first formula is over the set of all flat trees.
By the Stopping Time Theorem,
\begin{equation}\label{StoppingTime}
\|F_{\T}\|_{L_1} \lesssim \|F_{\Fl}\|_{L_1} \stackrel{\hbox{\tiny Cor. }\scriptscriptstyle{\ref{ToBesov}}}{\lesssim_{\eps}} \|F\|_{L_1}
\end{equation}
for each flat tree~$\T$. 

The inductive in~$n$ application of Remark~\ref{FlatLInfty} leads to the following lemma.
\begin{Le}\label{Inductive}
	Fix any~$\alpha > \kappa(p^{-1})$. Then, for any tree~$\T$ whose vertices are flat, whose root is~$\omega_0 \in \AF_{n_0}$, and for any~$n \geq n_0$, the inequality
	\begin{equation*}
	\Big\|\sum\limits_{\omega \in \T\cap\AF_n}F_{n+1}\chi_{\omega}\Big\|_{L_p} \lesssim e^{\alpha(n-n_0)}\|F_{n_0}\|_{L_p(\omega_0)}
	\end{equation*}
	is true provided~$\eps$ is sufficiently small.
\end{Le}

\begin{Cor}\label{SummationOverTree}
	Let~$\T$ be a tree whose vertices are flat and let~$\omega_0 \in \AF_{n_0}$ be its root. Then,
	\begin{equation*}
	\sum\limits_{n \geq n_0}m^{-\frac{p-1}{p}n}\Big\|\sum\limits_{\omega \in \T\cap\AF_n}f_{n+1}\chi_{\omega}\Big\|_{L_{p,1}} \lesssim_p \E \|F_{n_0}\|\chi_{\omega_0}
	\end{equation*}
	provided~$W$ satisfies the second structural condition.
\end{Cor}
\begin{proof}
	By Remark~\ref{StrictInequality}, we may choose~$p'$ to be slightly larger than~$p$ and such that~$\kappa(\frac{1}{p'})< \frac{p-1}{p}\log m$. Then, by the interpolation formula~$L_{p,1} = [L_{\frac{pp'}{2p'-p}},L_{p'}]_{(\frac12,1)}$ (we refer the reader to~\cite{BerghLofstrom} for the basics of interpolation theory) and Lemma~\ref{Inductive},
	\begin{multline*}
	\Big\|\sum\limits_{\omega \in \T\cap\AF_n}F_{n+1}\chi_{\omega}\Big\|_{L_{p,1}} \lesssim \Big\|\sum\limits_{\omega \in \T\cap\AF_n}F_{n+1}\chi_{\omega}\Big\|_{L_{p'}}^{\frac12}\Big\|\sum\limits_{\omega \in \T\cap\AF_n}F_{n+1}\chi_{\omega}\Big\|_{L_{\frac{pp'}{2p'-p}}}^{\frac12} \lesssim\\ e^{\alpha(n-n_0)}\|F_{n_0}\|_{L_{p'}(\omega_0)}^\frac12\|F_{n_0}\|_{L_{\frac{pp'}{2p'-p}}(\omega_0)}^\frac12 = e^{\alpha(n-n_0)}\|F_{n_0}\|_{L_p(\omega_0)}
	\end{multline*}
	for some~$\alpha \in (\kappa(\frac{1}{p'}),\frac{p-1}{p}\log m)$.
	
	Since this result holds for any~$n$, we may pass from capital~$F$ to the lowercase one:
	\begin{equation*}
	\Big\|\sum\limits_{\omega \in \T\cap \AF_{n}} f_{n+1}\chi_{\omega}\Big\|_{L_{p,1}} \lesssim e^{\alpha(n-n_0)}\|F_{n_0}\|_{L_p(\omega_0)}  =m^{n_{0}\frac{p-1}{p}} e^{\alpha(n-n_0)}\E \|F_{n_0}\|\chi_{\omega_0}.
	\end{equation*}
	It remains to estimate the sum of the geometric series:
	\begin{equation*}
	\sum\limits_{n \geq n_0}m^{-\frac{p-1}{p}n}\Big\|\sum\limits_{\omega \in \T\cap\AF_n}f_{n+1}\chi_{\omega}\Big\|_{L_{p,1}} \lesssim \sum\limits_{n \geq n_0}m^{-\frac{p-1}{p}(n-n_0)}e^{\alpha(n-n_0)}\E\|F_{n_0}\|\chi_{\omega_0} \lesssim \E\|F_{n_0}\|\chi_{\omega_0}.
	\end{equation*}
\end{proof}

\section{Proofs of Theorems~\ref{vanSchaftingen} and~\ref{Uncertainty}}

\subsection{Proof of inequality~\eqref{MainInequality}}
We want to estimate the sum
\begin{equation*}
\begin{split}
\sum\limits_{n=1}^{\infty} m^{-\frac{p-1}{p}n}\|f_n\|_{L_{p,1}} \lesssim \sum\limits_{n=0}^{\infty}m^{-\frac{p-1}{p}n}\Big\|\sum\limits_{\omega \in \Co \cap \AF_n}f_{n+1}\chi_\omega\Big\|_{L_{p,1}} + \\ \sum\limits_{n=0}^{\infty}m^{-\frac{p-1}{p}n}\Big\|\sum\limits_{\omega \in \Fl \cap \AF_n}f_{n+1}\chi_\omega\Big\|_{L_{p,1}}.
\end{split}
\end{equation*}
We will deal with the two sums separately. We need a slight improvement of the local embedding~\eqref{LocalEmbedding1} for the Lorentz scale to cope with the sum over convex atoms.
\begin{Le}\label{LocalEmbedding2}
	For any~$n$ and any~$\F_{n+1}$-measurable function~$g$, 
	\begin{equation*}
	\|g\|_{L_{p,1}} \lesssim m^{\frac{p-1}{p}n}\|g\|_{L_1}
	\end{equation*}
	with a uniform constant.
\end{Le}
\begin{proof}
	We use the fact that~$L_{p,1}$,~$p > 1$, is a ``normable'' space in the sense that there exists a norm~$\|\|\cdot\|\|$ on~$L_{p,1}$ that is equivalent to the classical quasi-norm on~$L_{p,1}$ (see~\cite{BerghLofstrom}). The norm of a linear operator acting from an~$\ell_1$-space to a normed space is attained at one of the extremal points of~$\ell_1$ (this is a consequence of the Minkowski inequality). In our case, each such extremal point is a characteristic function of an atom~$\omega'$ in~$\F_{n+1}$. We compute the norm:
	\begin{equation*}
	\|\|\chi_{\omega'}\|\| \lesssim \|\chi_{\omega'}\|_{L_{p,1}} \leq \|\chi_{\omega'}\|_{L_p} = m^{\frac{p-1}{p}(n+1)}\|\chi_{\omega'}\|_{L_1}.
	\end{equation*} 
	We have used the fact that on the set of characteristic functions of sets, the~$L_{p,q}$ norms for different~$q$ are comparable.
\end{proof}
It remains to combine Corollary~\ref{ToBesov} and Lemma~\ref{LocalEmbedding2} to get a half of the desired inequality~\eqref{MainInequality}:
\begin{equation*}
\sum\limits_{n=0}^{\infty}m^{-\frac{p-1}{p}n}\Big\|\sum\limits_{\omega \in \Co \cap \AF_n}f_{n+1}\chi_\omega\Big\|_{L_{p,1}} \lesssim \|F\|_{L_1}.
\end{equation*}

Let~$\T_1,\T_2,\ldots$ be all the trees in our graph, let~$\omega_1,\omega_2,\ldots$ be their roots, and let~$\omega_j \in \AF_{n_j}$. We use the triangle inequality (for that we need to pass from the standard quasi-norm in~$L_{p,1}$ to the norm~$\|\|\cdot\|\|$, use the triangle inequality, and then return back to the classical semi-norm):
\begin{equation*}
\begin{split}
\sum\limits_{n=0}^{\infty}m^{-\frac{p-1}{p}n}\Big\|\sum\limits_{\omega \in \Fl \cap \AF_n}f_{n+1}\chi_\omega\Big\|_{L_{p,1}} \lesssim  \sum\limits_{j}\sum\limits_{n=0}^{\infty}m^{-\frac{p-1}{p}n}\Big\|\sum\limits_{\omega \in \T_j \cap \AF_n}f_{n+1}\chi_\omega\Big\|_{L_{p,1}} \\ \stackrel{\hbox{\tiny Cor.~\ref{SummationOverTree}}}{\lesssim}\sum\limits_{j} \E\|F_{n_j}\|\chi_{\omega_j}.
\end{split}
\end{equation*}
Note that the parent of each atom~$\omega_j$ is convex (otherwise the tree~$\T_j$ is not maximal by inclusion). Thus, by Lemma~\ref{TriangleInequality},
\begin{equation*}
\sum\limits_{j} \E\|F_{n_j}\|\chi_{\omega_j} \lesssim \sum\limits_{j} \E(\|F_{n_j}\| - \|F_{n_j-1}\|)\chi_{(\omega_{j})^\uparrow} \stackrel{\scriptscriptstyle\eqref{StepwiseSplitting}}{\lesssim} \E\|F\|_{L_1}
\end{equation*}
since each atom~$(\omega_j)^\uparrow$ has at most~$m$ children and thus, arises in the sum at most~$m$ times. This finishes the proof of~\eqref{MainInequality}.

\subsection{Proof of inequality~\eqref{If}}

We need Lemma~$2$ in~\cite{SW}, which, in our setting, reads as follows. 
\begin{Le}
	Suppose that for some fixed number~$\gamma > 0$ and any collection~$C$ of disjoint atoms, the inequality
	\begin{equation}\label{OurFrostman}
	\sum\limits_{n}\sum\limits_{\omega \in C\cap \AF_n}\E|F_n\chi_{\omega}| \lesssim \bigg(\sum\limits_{n}\#\{\omega \in C\cap \AF_n\}m^{-n\beta}\bigg)^{\gamma}
	\end{equation} 
	holds true uniformly. Then,~$\dH \mu \geq \beta$.
\end{Le}
This version slightly differs from the original one, which was formulated for measures on~$\mathbb{R}^d$ instead of~$\mathbb{T}$. The proof works verbatim, because the Besicovitch and the Vitali Covering Theorems may be transferred to~$\mathbb{T}$. In fact, the geometry of balls in~$\mathbb{T}$ is much simpler than in~$\mathbb{R}^d$, because two balls either do not intersect, or one of them contains the other. In particular, this simple geometric property of balls immediately leads to the weak Besicovitch covering property, which in its turn, provides the proper analogs of Besicovitch and Vitali coverings. See~\cite{Rigot} for the details.

Let us first consider the case of a single tree, i.e. ~$F = F_{\T}$ and~$\mu = \mu_\T$. We need to prove inequality~\eqref{OurFrostman} with~$\beta$ arbitrarily close to~$1+\frac{\kappa'(1)}{\log m}$. We first combine H\"older's inequality with Lemma~\ref{Inductive}:
\begin{multline*}
\sum\limits_{\omega \in C\cap \AF_n}\E|F_n\chi_{\omega}| \leq \Big\|\sum\limits_{\omega \in C\cap\AF_n}F_{n}\chi_{\omega}\Big\|_{L_p}\Big(m^{-n}\#\{\omega \in C \cap \AF_n\}\Big)^{\frac{p-1}{p}} \lesssim\\ e^{\alpha(n-n_0)}m^{-n\frac{p-1}{p}}\Big(\#\{\omega \in C \cap \AF_n\}\Big)^{\frac{p-1}{p}}, \quad \alpha > \kappa(p^{-1}).
\end{multline*}
Therefore, by yet another H\"older's inequality,
\begin{multline*}
\sum\limits_{n}\sum\limits_{\omega \in C\cap \AF_n}\E|F_n\chi_{\omega}| \lesssim \sum\limits_{n}e^{\alpha(n-n_0)}m^{-n\frac{p-1}{p}}\Big(\#\{\omega \in C \cap \AF_n\}\Big)^{\frac{p-1}{p}} \leq\\ \Big(\sum\limits_{n}\big(e^{\alpha(n-n_0)}m^{(\beta-1)n\frac{p-1}{p}}\big)^p\Big)^{\frac{1}{p}}\Big(\sum\limits_{n}\#\{\omega \in C\cap \AF_n\}m^{-n\beta}\Big)^{\frac{p-1}{p}}.
\end{multline*}
We have proved~\eqref{OurFrostman} with~$\gamma = \frac{p-1}{p}$ provided the series in the first multiple converges, that is
\begin{equation*}
\alpha +\frac{p-1}{p}(\beta-1)\log m < 0.
\end{equation*}
Thus, it suffices to choose~$\alpha$ for which this inequality holds true. Since we can pick~$\alpha$ as close to~$\kappa(p^{-1})$ as we want, we are left with proving
\begin{equation}\label{InequalityForp}
\kappa(p^{-1}) + (1 - p^{-1})(\beta-1)\log m < 0.
\end{equation}
This inequality is true when~$\log m(\beta -1) <  \kappa'(1)$ and~$p$ is sufficiently close to~$1$. So, we first choose~$p$ in such a way that~\eqref{InequalityForp} is true, then choose~$\alpha$ sufficiently close to~$\kappa(p^{-1})$, and after that choose~$\eps$ in Lemma~\ref{InequalityForNon-Intensive2}. We have proved that~$\dH \mu_\T \geq 1+\frac{\kappa'(1)}{\log m}$ for a single tree~$\T$.

Before we pass to the case of a general measure, we need to get more information about a single tree. This requires yet another tiny portion of combinatorics. We call a convex atom~$\omega$ a fruit of~$\T$ provided~$\omega^\uparrow \in \mathcal{T}$. We call a point~$x\in \mathbb{T}$ a leaf of~$\T$ provided there is an infinite sequence~$\{\omega_n\}_{n}$ such that~$\omega_n\in \T\cap\AF_n$ and~$x\in\omega_n$. In other words, the leafs are the points that correspond to infinite paths inside~$\T$. It is easy to see that
\begin{equation*}
\omega_0 = \Big(\bigcup_{\genfrac{}{}{0pt}{-2}{\scriptscriptstyle \omega\ \hbox{\tiny is a}}{\hbox{\tiny fruit of }\scriptscriptstyle \T}}\omega\Big) \cup \{x\mid x\ \hbox{is a leaf of }\T\},
\end{equation*}
where~$\omega_0$ is the root of~$\T$.
So, we split the measure~$\mu_\T$ into two parts~$\mu_{\T,c}$ and~$\mu_{\T,s}$ that are the restrictions of~$\mu_\T$ to the union of all fruits and to the set of leafs, correspondingly. 

Now we are ready to cope with the case of a general measure~$\mu \in \W$. We pick some~$\eps$ to be chosen later and split~$\mu$ into convex part~$\mu_{\Co}$ and flat part~$\mu_{\Fl}$. Note that the convex part is an~$L_1$ function by Corollary~\ref{ToBesov}. So, it suffices to deal with the flat part. For any tree~$\T$ in our decomposition,
\begin{equation*}
\|\mu_{\T,c}\|_{L_1} = \sum\limits_{\genfrac{}{}{0pt}{-2}{\scriptscriptstyle \omega\ \hbox{\tiny is a}}{\hbox{\tiny fruit of}\; \scriptscriptstyle \T}} \|\mu_{\T,c}\chi_{\omega}\|_{L_1} \stackrel{\scriptscriptstyle \omega\ \hbox{\tiny is convex}}{\lesssim} \sum\limits_{n=0}^{\infty}\sum\limits_{\genfrac{}{}{0pt}{-2}{\scriptscriptstyle\omega \in \AF_n}{\scriptscriptstyle \omega\; \hbox{\tiny fruit of }\; \T}}\E(\|F_{n+1}\| - \|F_n\|)\chi_{\omega}.
\end{equation*}
So, if~$\T_1,\T_2,\ldots$ are all the trees of our decomposition, then~$\sum_j \mu_{\T_j,c}$ is an~$L_1$-function by formula~\eqref{StepwiseSplitting} and the fact that different trees have different fruits. As for the singular parts, we note that they have disjoint supports. Thus,
\begin{equation*}
\dH\sum\limits_j\mu_{\T_j,s} \geq \inf_j\dH\mu_{\T_j,s},
\end{equation*}
which is not less than any number larger than~$1+\frac{\kappa'(1)}{\log m}$, provided~$\eps$ is sufficiently small, as we proved before.

\subsection{Proof of sharpness in Theorem~\ref{Uncertainty}.}
In fact, the example showing sharpness was constructed in~\cite{Eggleston}. We only provide some details. We start with the Eggleston formula. We may represent points~$x \in [0,1]$ in~$m$-adic system, i.e.
\begin{equation*}
x = \sum\limits_{0}^{\infty}c(x,n)m^{-n},
\end{equation*}
where~$c(x,m) \in [0..m-1]$ are the ``digits'' of~$x$. We should also assume that for any~$x$ the sequence~$c(x,m)$ is not identical~$m-1$ eventually. Let~$p_1,p_2,\ldots,p_m$ be the probability distribution on~$[0..m-1]$. Consider the set
\begin{equation*}
S = \Big\{x\in [0,1]\,\Big|\;\forall j = 0,1,\ldots,m-1 \quad \lim\limits_{n\to \infty}\frac{\#\{k\leq n\mid c(x,k) = j\}}{n} = p_j\Big\}.
\end{equation*}
Eggleston's formula (see~\cite{Eggleston}) says
\begin{equation}\label{Eggleston'sFormula}
\dH S = -\frac{\sum_{j=1}^m p_j \log p_j}{\log m},
\end{equation}
where we measure the Hausdorff dimension of~$S$ as of a subset of~$[0,1]$ 
In fact, we need the inequality~$\dH S \leq -\frac{\sum_{j=1}^m p_j \log p_j}{\log m}$ only. This inequality is simpler than the reversed one and the argument of~\cite{Eggleston} can be transferred to the case where we replace~$[0,1]$ with~$\mathbb{T}$ with ease (what we need here is that the~$m$-adic subintervals of~$[0,1]$ have the same diameter as the corresponding cylinders in~$\mathbb{T}$).

We note that the infimum in formula~\eqref{FormulaForDerivativeOfKappa} is attained at some~$v\in V$ (since~$t\log t$ is a continuous function). We repeat the construction of Lemma~\ref{CounterExampleToHLS} by writing
\begin{equation*}
h_{n+1} = \sum\limits_{\omega \in \AF_n}\J_{\omega}[v]\cdot\chi_{\omega}
\end{equation*}
and
\begin{equation*}
F_n = \prod\limits_{i=1}^{n} (1+h_i).
\end{equation*}
Then~$F$ is an~$L_1$ martingale. Let~$\mu$ be its limit measure. Put~$p_j = \frac{1+v_j}{m}$ in the definition of set~$S$. It suffices to prove that~$\mu(S) = 1$ since~$\mu \in\W$ and
\begin{equation*}
\dH S \stackrel{\scriptscriptstyle{\eqref{Eggleston'sFormula}}}{=} -\frac{\sum_{j=1}^m p_j \log p_j}{\log m} = -\frac{\sum_{j=1}^m (1+v_j) \big(\log (1+v_j)-\log m\big)}{m\log m} = 1+\frac{\kappa'(1)}{\log m}
\end{equation*} 
since~$v$ is the extremal vector in~\eqref{FormulaForDerivativeOfKappa}. To prove that~$\mu(S) = 1$, we consider yet another probability space~$(\mathbb{T},\mu)$ (clearly,~$\mu$ is a probability measure). Consider the random variables
\begin{equation*}
\xi_{n,j}(x) = \chi_{\{c(x,n) = j\}}.
\end{equation*}
Note that~$\E \xi_{n,j} = p_j$ and these random variables are independent (when~$j$ is fixed). Thus, by the Kolmogorov Strong Law of Large Numbers,
\begin{equation*}
\frac{\#\{k\leq n\mid c(x,k) = j\}}{n} = \frac{\sum_{k \leq n}\xi_{n,j}}{n} \to p_j
\end{equation*}
for any~$j$ almost surely (with respect to~$\mu$). Therefore,~$\mu(S) = 1$ and Theorem~\ref{Uncertainty} is proved.

\section{Comparison with Roginskaya--Wojciechowski conjecture}
Let us consider the translation invariant case again (the space~$W$ is given by formula~\eqref{ShiftInvariant}). We may try to make a conjecture similar to Conjecture~$1$ in~\cite{RW}.
\begin{Conj}
	If~$\dim W_{\gamma} = k$ for all~$\gamma \in G \setminus \{0\}$, then~$\dH \mu \geq \frac{m-k}{m}$.
\end{Conj}
It is easy to see that such a conjecture cannot hold: if~$\cap_{\gamma \ne 0} W_{\gamma} \ne \{0\}$, then the second structural assumption is violated, and we can find a delta-measure in~$\W$. However, we can easily deduce  the statement below from Theorem~\ref{Uncertainty}.
\begin{Cor}\label{antisymmetry}
	For any~$a\in\mathbb{R}^\ell$, consider the set~$W^{-1}(a)$ given by the rule:
	\begin{equation*}
	W^{-1}(a) = \{\gamma \in G \mid a \in W_\gamma\cap W_{-\gamma}\}.
	\end{equation*}
	If, for any~$a \in\mathbb{C}^\ell \setminus \{0\}$, the set~$W^{-1}(a)$ is contained in a subgroup of size~$K$, then~$\dH\mu \geq 1-\frac{\log K}{\log m}$ for any~$\mu \in \W$.
\end{Cor}

We can regard it as a model of the following conjecture from~\cite{AW}.

\begin{Def} We say that a non-constant function~$\Omega\colon S^{d-1} \to G(d,1)$ is $k$-antisymmetric \\ ($k=0,1,...,d-1$), if
\begin{equation*}
\bigcap_{\xi \in S^{d-1}\cap L} \Omega(\xi)= \{0\}
\end{equation*}
for each $(k+1)$-dimensional subspace  $L \subset \R^{d}$. Denote
	\begin{equation*}
	a(\Omega) =\min \{k\mid \Omega \ \ \hbox{is~$k$-antisymmetric}\}.
	\end{equation*}
\end{Def}
\begin{Conj}
	\label{ascon}
	If $\mu \in \mathfrak{BV}$~for a smooth non-constant function $\Omega$, then
	\begin{equation*}
	\dH(\mu) \geq d-a(\Omega).
	\end{equation*}
\end{Conj}

\section{Analogs of trace theorems}
We have studied the~$L_p$-regularity properties of~$\I_{\alpha}[F]$, where~$F\in \W$. What about its continuity properties? Say, when are we able to define the trace of~$\I_{\alpha}[F]$ on a set~$E$ of fractional dimension? As usual, we will find the inspiration in the classical real-variable case. It is known (see, e.g.~\cite{Mazja}) that~$\dot{W}_1^1(\mathbb{R}^d)$ embeds into~$L_p(\nu)$ if and only if~$\nu(B_r(x))^{\frac{1}{p}} \lesssim r^{d-1}$ for any ball~$B_r(x)$ (the symbol~$B_r(x)$ denotes the Euclidean ball of radius~$r$ centered at~$x$). In particular, a Sobolev function with~$L_1$ derivatives has meaningful traces on~$(d-1)$-dimensional surfaces. We signalize that~$\I_{1}[\nabla f]$ does not necessarily have a meaningful trace on a hyperplane. In fact, this is the first time when our statements really ``feel'' the difference between~$\I_1[\nabla f]$ and~$f$.

\begin{Def}
	We say that a measure~$\nu$ with bounded variation on~$\mathbb{T}$ satisfies the~$(\alpha,p)$ Frostman condition provided
	\begin{equation}\label{Frostman}
	\forall n \geq 0\quad \forall \omega \in \AF_n \quad (\nu(\omega))^{\frac{1}{p}} \lesssim m^{(\alpha - 1)n}.
	\end{equation}
\end{Def}
\begin{Rem}
	The~$(\alpha,p)$ Frostman condition is necessary for the continuity of the mapping~$\I_{\alpha}\colon \W \to L_p(\nu)$.
\end{Rem}
The theorem below is a straightforward generalization of Theorem~\ref{vanSchaftingen}, so we leave its proof to the reader.
\begin{Th}\label{TraceTheoremTrivial}
	Let~$p > 1$. If~$\nu$ satisfies the~$(\alpha,p)$ Frostman condition and~$W$ satisfies the second structural assumption, then the mapping~$\I_{\alpha}\colon \W \to L_p(\nu)$ is continuous.
\end{Th}
The limiting case~$p=1$ is more involved. 
\begin{Th}\label{TraceTheorem}
	Assume~$\nu$ satisfies the~$(\alpha,1)$ Frostman condition and
	\begin{equation}\label{LowerBoundForAlpha}
	\alpha > -\frac{\kappa'(1)}{\log m}.
	\end{equation}
	Then, the mapping~$\I_{\alpha}\colon \W \to L_1(\nu)$ is continuous.
\end{Th}
It seems that the limit case of equality in~\eqref{LowerBoundForAlpha} is delicate in the already mentioned sense: maybe, we can replace the positive operator~$\I_{\alpha}$ by an operator of the same order that has more cancellations to have the continuity.

We are able to prove sharpness of~\eqref{LowerBoundForAlpha} only in a very special case where the function~$\kappa$ is linear (note that the classical Sobolev spaces satisfy this condition).
\begin{Le}\label{SufficiencyForTrace}
	Let the function~$\kappa$ be linear on the interval~$[0,1]$. If the mapping~$\I_{\alpha}\colon \W \to L_1(\nu)$ is continuous, then~\eqref{LowerBoundForAlpha}.
	\end{Le}

\begin{proof}[Proof of Theorem~\ref{TraceTheorem}]
	The proof follows the lines of the proof of Theorem~\ref{vanSchaftingen} with slight modifications. As usual, we will estimate~$\|\I_{\alpha}[F_{\Co}]\|_{L_1(\nu)}$ and~$\|\I_\alpha[F_{\Fl}]\|_{L_1(\nu)}$ separately. Here is the estimate for the convex part:
	\begin{multline*}
	\|\I_\alpha[F_{\Co}]\|_{L_1(\nu)} \leq \sum\limits_{n\geq 0}m^{-\alpha n}\|(f_{\Co})_{n+1}\|_{L_1(\nu)} = \sum\limits_{n\geq 0}m^{-\alpha n}\Big(\sum\limits_{\omega \in \Co\cap \AF_n}|f_{n+1}|_{\omega}|\nu(\omega)\Big) \stackrel{\scriptscriptstyle{\eqref{Frostman}}}{\leq}\\ \sum\limits_{n\geq 0}m^{-n}\Big(\sum\limits_{\omega \in \Co\cap \AF_n}|f_{n+1}|_{\omega}|\Big) \stackrel{\hbox{\tiny Cor.~\ref{ToBesov}}}{\leq} \|F\|_{L_1}.
	\end{multline*}
	Similar to the proof of Theorem~\ref{vanSchaftingen}, the bound
	\begin{equation}\label{IndividualTree}
	\|\I_{\alpha}[F_{\T}]\|_{L_1(\nu)} \lesssim m^{-n_0}|F_{n_0}(\omega_0)|
	\end{equation}
	for any tree of flat atoms~$\T$ with the root~$\omega_0 \in \AF_{n_0}$, leads to the wanted estimate for the flat part. So, it remains to prove~\eqref{IndividualTree}.
	
	Let~$\{N_n\}_n$ be the martingale that represents~$\nu$. We interpolate the~$(\alpha,1)$-Frostman condition
	\begin{equation*}
	\|N_n\|_{L_{\infty}} \lesssim m^{\alpha n}
	\end{equation*} 
	with the trivial bound~$\|N_n\|_{L_1(\omega_0)}\lesssim m^{(\alpha - 1)n_0}$ to get
	\begin{equation}\label{InterpolatoryEstimateForMeasure}
	\|N_n\|_{L_{\frac{p}{p-1}}(\omega_0)} \lesssim m^{\frac{p-1}{p}(\alpha -1)n_0 + \frac{\alpha n}{p}}.
	\end{equation}
	
	Now we can prove~\eqref{IndividualTree}:
	\begin{multline*}
	\|\I\alpha[F_{\T}]\|_{L_1(\nu)} \leq \sum\limits_{n \geq n_0}m^{-\alpha n}\sum\limits_{\omega \in \T\cap \AF_n}|f_{n+1}|_\omega|\nu(\omega) =\\  
	\sum\limits_{n \geq n_0}m^{-\alpha n}\E\Big(\sum\limits_{\omega \in \T\cap \AF_n}|f_{n+1}|_\omega|\chi_{\omega}N_n\Big) \leq \\ 
	\sum\limits_{n \geq n_0}m^{-\alpha n}\Big\|\sum\limits_{\omega \in \T\cap \AF_n}|f_{n+1}|_\omega|\chi_{\omega}\Big\|_{L_p(\omega_0)}\|N_n\|_{L_{\frac{p}{p-1}}(\omega_0)} \stackrel{\stackrel{\hbox{\tiny Lem.~\ref{Inductive},}}{\scriptscriptstyle \beta > \kappa(p^{-1})}}{\lesssim}\!\!\!\!\!\!\!\!_{\eps}  \\ 
	\sum\limits_{n\geq n_0} m^{-\alpha n} e^{\beta(n-n_0)}\|F_{n_0}\|_{L_p(\omega_0)} \|N_n\|_{L_{\frac{p}{p-1}}(\omega_0)} \stackrel{\scriptscriptstyle \eqref{InterpolatoryEstimateForMeasure}}{\lesssim}\\  \sum\limits_{n\geq n_0} m^{-\alpha n} e^{\beta(n-n_0)}\|F_{n_0}\|_{L_p(\omega_0)} m^{\frac{p-1}{p}(\alpha -1)n_0 + \frac{\alpha n}{p}} \lesssim \\  m^{-(1-\frac{1}{p})n_0}\|F_0\|_{L_p(\omega_0)}\sum\limits_{n\geq n_0} m^{-(1-\frac{1}{p})\alpha(n-n_0)}e^{\beta(n-n_0)}.
	\end{multline*} 
	In order for the geometric series to converge, we should find~$\beta > \kappa(p^{-1})$ such that~$\beta < (1-\frac{1}{p})\alpha \log m$ or, what is the same,
	\begin{equation*}
	\alpha > \frac{\kappa(p^{-1})}{(1-\frac{1}{p})\log m}.
	\end{equation*}
	For sufficiently small~$p$, this follows from~\eqref{LowerBoundForAlpha}.
\end{proof}

\begin{proof}[Proof of Lemma~\ref{SufficiencyForTrace}]
	Let~$v \in V$ be such that
	\begin{equation*}
	\Big(\frac{1}{m}\sum\limits_{j=1}^m |1+v_j|^2\Big)^\frac12 = e^{\kappa(\frac12)},
	\end{equation*}
	there exists non-zero~$a\in\mathbb{R}^\ell$ such that~$v\otimes a \in W$ and~$v_j \geq -1$ for all~$j$. We construct the multiplicative martingale~$F$ similar to the one used in the proof of Lemma~\ref{CounterExampleToHLS}:
	\begin{equation*}
	F_n = \prod_{j \leq n}(1+h_n), \quad h_n = \sum\limits_{\omega \in \AF_n}\J_\omega[v]\chi_{\omega}.
	\end{equation*}
	We define the measure~$\nu$ by the formula
	\begin{equation*}
	\nu = F_0+\I_{\gamma}[F],
	\end{equation*}
	where~$\gamma$ is a non-negative number to be chosen later. Since~$F \geq 0$,~$\nu$ is also non-negative. Let~$\{N_n\}_n$ be the martingale that represents~$\nu$. Then,
	\begin{equation*}
	\begin{split}
	\|N_{n_0}\|_{L_{\infty}} \leq \sum\limits_{n < n_0} \|m^{-\gamma n}f_{n+1}\|_{L_{\infty}} \leq \sum\limits_{n < n_0} m^{-\gamma n}e^{n\kappa(0)} = \sum\limits_{n < n_0} e^{n(\kappa(0) - \gamma\log m)} \\ \lesssim m^{(\frac{\kappa(0)}{\log m} - \gamma)n_0},
	\end{split}
	\end{equation*}
	provided~$\gamma < \frac{\kappa(0)}{\log m}$. Thus,~$\nu$ satisfies~$(\alpha,1)$-Frostman condition with~$\alpha = \frac{\kappa(0)}{\log m} - \gamma$. 
	
	Now we estimate~$\|\I_{\alpha} [F]\|_{L_1(\nu)}$ from below:
	\begin{equation*}
	\begin{split}
	\|\I_{\alpha} [F]\|_{L_1(\nu)} \gtrsim \E \I_{\alpha} [F] I_{\gamma}[F] = \sum\limits_{n >0}m^{-(\alpha + \gamma)n}\E|f_n|^2 \asymp 
	\sum\limits_{n > 0}m^{-(\alpha +\gamma)n}e^{2\kappa(\frac12)} \\ = \sum\limits_{n > 0}e^{2\kappa(\frac12) - (\alpha+\gamma)\log m}.
	\end{split}
	\end{equation*}
	Since~$2\kappa(\frac12) = \kappa(0)$, this series diverges, and~$\|\I_{\alpha} [F]\|_{L_1(\nu)}$ is infinite. We note that~$F\otimes a \in \W$, and so,~$\I_{\alpha} \colon \W \to L_1(\nu)$ is not bounded. Since~$\kappa(0) = -\kappa'(1)$, we may take any~$\alpha \in (0,-\frac{\kappa'(1)}{\log m}]$.
\end{proof}

Rami Ayoush

Institute of Mathematics, Polish Academy of Sciences.

rayoush@impan.pl.

\medskip

Dmitriy M. Stolyarov

St. Petersburg State University, 14th line 29B, Vasilyevsky Island, St. Petersburg, Russia.

St. Petersburg Department of Steklov Mathematical Institute, Russian Academy of Sciences.

d.m.stolyarov@spbu.ru.

\medskip

Michal Wojciechowski

Institute of Mathematics, Polish Academy of Sciences.

miwoj.impan@gmail.com.

\end{document}